\documentclass[11pt, leqno]{amsart}
\usepackage{amsmath,amsthm,amsfonts,amssymb,amscd,tikz,amscd}

\DeclareMathOperator\coker{coker}

\newcommand{\bbC}{{\mathbb C}}

\newcommand{\bbZ}{{\mathbb Z}}

\newcommand{\cO}{{\mathcal O}}

\newcommand{\an}{\mathrm{an}}
\newcommand{\tor}{\mathrm{tor}}
\newcommand{\Spec}{\mathrm{Spec}}
\newcommand{\Supp}{\mathrm{Supp}}
\newcommand{\Hom}{\mathrm{Hom}}
\newcommand{\Ext}{\mathrm{Ext}}
\newcommand{\coev}{\mathrm{coev}}
\newcommand{\ev}{\mathrm{ev}}
\newcommand{\id}{\mathrm{id}}
\newcommand{\free}{\mathrm{free}}
\newcommand{\tf}{\mathrm{tf}}
\newcommand{\fd}{\mathrm{fd}}
\newcommand{\rk}{\mathrm{rk}}
\newcommand{\Pic}{\mathrm{Pic}}
\newcommand{\coh}{\mathrm{Coh}}
\newcommand{\Auteq}{\mathrm{Auteq}}
\newcommand{\Aut}{\mathrm{Aut}}
\newcommand{\GL}{\mathrm{GL}}
\newcommand{\tmod}{\mathrm{mod}}

\newtheorem{thm}{Theorem}[section]

\newtheorem{example}[thm]{Example}
\newtheorem{theorem}[thm]{Theorem}
\newtheorem{cor}[thm]{Corollary}

\newtheorem{prop}[thm]{Proposition}
\newtheorem{lemma}[thm]{Lemma}

\newtheorem{conj}[thm]{Conjecture}

\newtheorem{remark}[thm]{Remark}

\newtheorem{defn}[thm]{Definition}

\newtheorem*{claim*}{Claim}
\numberwithin{equation}{section}

\begin{document}

\title{Quantum Elliptic Curves I: Algebraic Case}

\author{Michael Larsen}
\address{Department of Mathematics, Indiana University, Bloomington, IN 47405, USA}
\email{mjlarsen@iu.edu}
\thanks{Michael Larsen was partially supported by NSF grant DMS-2401098 and the Simons Foundation.}

\author{Valery Lunts}
\address{Department of Mathematics, Indiana University, Bloomington, IN 47405, USA\\
National Research University Higher School of Economics, Moscow, Russia}
\email{vlunts@iu.edu}

\begin{abstract} A complex elliptic curve $E$ can be defined as the quotient of the analytic space $\bbC^*$ by a discrete action of the cyclic group $q^{\bbZ}$ for $\vert q\vert \neq 1$. We study the boundary case when $\vert q\vert =1$, which leads to the notion of a {\rm quantum} elliptic curve and a conjectural equivalence of categories that one might call a {\rm noncommutative} GAGA.
\end{abstract}

\maketitle

\section{Introduction}

\subsection{Motivation} Fix $q\in \bbC ^*$ and let $q^{\bbZ}\subset \bbC ^*$ be the cyclic subgroup
generated by $q$. Assume that $\vert q\vert \neq 1$. Then the action of $q^{\bbZ}$ on $\bbC ^*$ is discrete and the quotient $\bbC ^*/q^{\bbZ}$ is an elliptic curve $E_q$. The quotient map
$$\pi :\bbC ^*\to E_q$$
induces an equivalence of categories of analytic coherent sheaves
\begin{equation}\label{first equiv} \pi ^* :\coh ^{\an}_{E_q}\to \coh _{\bbC ^*}^{\an,q^\bbZ},
\end{equation}
where $\coh _{\bbC ^*}^{an, q^\bbZ}$ is the corresponding category of $q^{\bbZ}$-{\it equivariant} analytic coherent sheaves on $\bbC ^*$.
By GAGA, the analytic category $\coh ^{\an}_{E_q}$ is equivalent to the category $\coh _{E_q}$ of algebraic coherent sheaves on $E_q$. Thus the equivalence \eqref{first equiv} combined with the classical GAGA may be considered as the {\it algebraization} of the analytic category $\coh _{\bbC ^*}^{an,q^\bbZ}$:
\begin{equation}\label{second equiv} \coh _{E_q}\to \coh _{\bbC ^*}^{an,q^\bbZ}.
\end{equation}

Let $\cO$ be the algebra of analytic functions on $\bbC^*$. The group $q^{\bbZ}$ acts on $\cO$ and we let  $A_q^{\an}$ be the corresponding crossed product algebra. This algebra consists of non-commutative Laurent polynomials in a variable $\sigma$ over the ring $\cO$  where $\sigma f(z) = f(qz) \sigma$.

In \cite{SV}, Soibelman and  Vologodsky showed that the sheaves in the category $\coh _{\bbC ^*}^{an, q^\bbZ}$ are actually globally coherent and hence this category is equivalent to the category of $A_q^{\an}$-modules which are finitely presented as $\cO$-modules.

When $|q|=1$ and $q$ is not a root of unity, the orbits of $q^{\bbZ}$ on $\bbC^\ast$ are no longer discrete, so the geometric quotient $E_q$ no longer makes sense. But the algebra $A_q^{\an}$ is defined and one can still consider the category of $A_q^{\an}$-modules which are finitely presented as $\cO$-modules. Let us denote this category $\mathcal{M}_q^{\an}$.
Soibelman and  Vologodsky called it the category of (analytic) coherent sheaves on the quantum elliptic curve.

One may ask for an algebraization of the analytic category $\mathcal{M}_q^{\an}$ in case $\vert q\vert =1$ which would be the analogue of the equivalence \eqref{second equiv}. We propose a conjectural answer.
Namely let $A=\bbC [z,z^{-1}]$ be the algebra of Laurent polynomials. The group $q^{\bbZ}$ acts on $A$ and we let $A_q$ be the corresponding crossed product algebra, i.e. $A_q=\langle z,z^{-1},\sigma , \sigma ^{-1}\rangle /(\sigma z=qz\sigma)$. Let $\mathcal{M}_q$ be the abelian category of $A_q$-modules which are finitely generated as $A$-modules. The natural embedding of algebras $A\hookrightarrow \cO$ extends to the embedding
$A_q\hookrightarrow A_q^{\an}$ and we obtain the functor of extension of scalars
\begin{equation}\label{conj equiv}F:\mathcal{M}_q\to \mathcal{M}_q^{\an},\quad M\mapsto \cO \otimes _AM.
\end{equation}

\begin{conj}\label{conj} (i) Suppose that $\vert q\vert =1$ and $q$ is not a root of $1$. Then the functor $F$ is full and faithful.

(ii) Suppose in addition that there exists $L>0$ such that for every $0\neq n\in \bbZ$ we have $\vert q^n-1\vert >L^n$.
Then $F$ is an equivalence of categories.
\end{conj}

One may consider the equivalence \eqref{second equiv} and the conjectural equivalence \eqref{conj equiv} as two instances of {\it noncommutative} GAGA. If $\vert q\vert \neq 1$, then the algebraic structure on $E_q$ plays the role. But in the limit case $\vert q\vert =1$ the algebraic structure on $\bbC ^*$ appears.

We will address Conjecture \ref{conj} in the next paper \cite{LaLu}.

\subsection{The present paper}
In this paper, we study the {\it algebraic} quantum elliptic curve. That is we assume that $q\in \bbC ^*$ is not a root of unity and study the algebra $A_q$. More precisely, we study the corresponding category $\mathcal{M}_q$ which we consider to be the category of (algebraic) {\it coherent sheaves}
on the quantum elliptic curve.

We establish a number of structural properties of the algebra $A_q$ and of the category $\mathcal{M}_q$. In particular we study the cohomology of objects in $\mathcal{M}$. For the {\it structure sheaf $\cO \in \mathcal{M}_q$} we have $h^0(\cO)=h^1(\cO )=1$, as in the case of a usual elliptic curve. Also the Picard group $Pic(\mathcal{M}_q)$ is isomorphic to the quotient $\bbC ^*/q^{\bbZ}$, as in the classical case.

We prove a ``Riemann-Roch'' theorem and a ``Serre duality'' theorem. Both look rather different from their classical counterparts. For instance, a non-trivial line bundle cannot have a non-zero section in the algebraic quantum setting. This is in agreement with the analytical picture.  Indeed, in the classical analytic setting, a non-zero section of
such a line bundle has a Laurent series which is unbounded on both sides because it is formally that of a theta function. Whether such a series can actually converge depends on whether the degree of the line bundle is positive or negative and on whether $|q|<1$ or $|q|>1$. When $|q|=1$, there are no non-zero sections.

For non-trivial line bundles $L\in \mathcal{M}_q$, $h^1(L)$ is the absolute value of the degree of $L$. Thus, ``Serre duality'' (which is no longer a duality theorem but just a comparison of dimensions) compares $h^i(,L)$ to $h^j(L^{-1})$ for $i=j$ rather than $i+j=1$.

We also establish the symmetry property of the Euler form of the category $\mathcal{M}_q$: for $F,G\in \mathcal{M}_q$ we prove that $\chi (F,G)=\chi (G,F)$, where
$$\chi (F,G)=\Ext^0(F,G)-\Ext^1(F,G)$$
We notice that the same symmetry property holds in the category of holonomic modules over the Weyl algebra \cite{Sa}. This underlines the analogy between the Weyl algebra and its ``multiplicative analogue" $A_q$.

Also in contrast with the classical case, every object in $\mathcal{M}_q$ has finite length and there exist simple objects whose rank is greater than $1$.

In the last Section \ref{addit results} we discuss some applications to the category $\mathcal{M}_q$ of some results of A. Elagin \cite{El}

We gratefully acknowledge useful conversations with A.\,Elagin, M.\,Kontsevich, T.\,Stafford and Yu.\,Berest.

\section{The algebra $A_q$}

\subsection{The algebra $A_q$}\label{intro of algebra} Fix $q\in \bbC ^*$ such that  $q^n\neq 1$ for any $0\neq n\in \bbZ$. Let $R = \bbC[F_2]$ denote the complex group algebra of
the free abelian group on generators $\sigma$ and $z$.  Consider the noncommutative algebra
$$A_q:= R / R(\sigma z-qz\sigma)R.$$
It is clear that every group element $g\in F_2$  is congruent modulo $R(\sigma z-qz\sigma)R$ to a monomial of the form $z^m \sigma^n$, and that these monomials are linearly independent over $\bbC$ modulo
$R(\sigma z-qz\sigma)R$.
We therefore think of $A_q$ as the ring of Laurent polynomials in variables $\sigma$ and $z$ whose commutation rule is $\sigma z = qz\sigma$.
It contains the following two  commutative subalgebras:
$$A=\bbC [z,z^{-1}],\quad S=\bbC [\sigma ,\sigma ^{-1}].$$
We denote by $K(A)$ the fraction field of $A$.

Every element $x\in A_q$ can be expressed uniquely either as
$$x=x_m(z)\sigma ^m+x_{m+1}(z)\sigma ^{m+1}+\cdots+x_{n}(z)\sigma ^n,\quad \text{where}\quad x_i(z)\in A$$
or as
$$x=x_k(\sigma )z ^k+x_{k+1}(\sigma )z^{k+1}+\cdots+x_{l}(\sigma )z ^l,\quad \text{where}\quad x_i(\sigma )\in S.$$
Assuming that $x_m(z),x_n(z),x_k(\sigma ),x_l(\sigma)\neq 0$ we call $n-m$ (resp. $l-k$) the $\sigma $-degree of $x$, denoted $\deg _\sigma x$ (resp. the $z$-degree of $x$, denoted $\deg _zx$). These degrees satisfy
$$\deg _\sigma (xy)=\deg _\sigma x+\deg _\sigma y\quad \text{and}
\quad \deg _z (xy)=\deg _z x+\deg _z y.$$

Note that in the above notation we also have
$$x=\sigma ^mx_m(q^{-m}z)+\sigma ^{m+1}x_{m+1}(q^{-m-1}z)+\cdots+\sigma ^nx_n(q^{-n}z).$$

The algebra $A_q$ has a distinguished automorphism of order $4$ (the ``Fourier transform'') given by
$$\phi (z)=\sigma,\quad \phi (\sigma )=z^{-1}.$$

It also has a distinguished anti-automorphism $\epsilon :A_q\to A_q$, defined by $\epsilon (z)=z$ and $\epsilon (\sigma )=\sigma ^{-1}$. So if
$w=\sum _iw_i(z)\sigma ^i$, then
$${}^\epsilon w=\sum _i\sigma ^{-i}w_i(z).$$

The units in the algebra $A_q$ are the monomials $cz^m\sigma ^n$, $m,n\in \bbZ$, $c\in \bbC ^*$.

Given $x,y\in A_q$ we will denote their product either by $xy$ or by $x\cdot y$.

\begin{lemma}\label{lemma list of basic prop} The algebra $A_q$ has the following properties

(1) $A_q$ is a domain (has no zero divisors).

(2) $A_q$ is simple (has no nontrivial two-sided ideals).

(3) $A_q$ is left and right Noetherian.

(4) $A_q$ has global homological dimension 1.
\end{lemma}

\begin{proof} (1) is obvious;

(2) is \cite[Thm 1.17]{GW};

(3) is \cite[Cor  1.15]{GW};

(4) is \cite[Ch 7, Thm 5.5]{MR}.
\end{proof}

\begin{remark}\label{remark that aq is not a pid} There exist non-principal
left (and right) ideals in $A_q$. An example of such an ideal is given in Example \ref{ex of two gener}.
\end{remark}

\begin{defn}\label{defn of o} We denote by $\mathcal{O}$ the left $A_q$-module $A$ with its $A_q$-module structure determined by $\sigma \cdot 1=1$. So $\mathcal{O}\simeq A_q/A_q\cdot (\sigma -1)$.
\end{defn}

\subsubsection{Division with remainder in $A_q$}

\begin{lemma} \label{lemma on div with remainder} ($\sigma$-division) Let $w=w(z,\sigma ),\ r=r(z,\sigma )$ be nonzero elements in $A_q$. Assume that $\deg _\sigma w\leq \deg _\sigma r$. Then there exist nonzero $g=g(z)\in A$ and $h=h(z,\sigma )\in A_q$ such that
\begin{equation}\label{eq div with rem}
\deg_\sigma (gr-hw)<\deg _\sigma w.
\end{equation}
In case the coefficient of the highest (or the lowest) power  of $\sigma$ in $w$ is a unit, one may choose $g (z)$  to be a unit as well.
\end{lemma}

\begin{proof} Write $w$ and $r$ as Laurent polynomials in $\sigma$
$$w=w_k(z)\sigma ^k + \mathrm{lower},\quad
r=r_l(z)\sigma ^l+ \mathrm{lower},$$
where $w_k(z),r_l(z)\neq 0$. To find $g$ and $h$ as required we first  take $$s=s(z,\sigma ):=w_k(q^{l-k}z) r-r_l(z)\sigma ^{l-k}w.$$
Then $\deg _\sigma s<\deg _\sigma r$.
If $\deg _{\sigma}s<\deg _\sigma w$, then we are done. Otherwise,
by induction on the difference $\deg _\sigma r-\deg _\sigma w$ we may assume that the assertion holds for the pair $(s,w)$, hence it also
holds for the pair $(r,w)$.

One could alternatively write
$$w=w_k(z)\sigma ^k + \mathrm{higher},\quad
r=r_l(z)\sigma ^l+ \mathrm{higher},$$
where $w_k(z),r_l(z)\neq 0$. Then the same algorithm applies again.

The last assertion of the lemma is now clear.
\end{proof}

The proof of the next lemma is completely analogous to the proof of
Lemma \ref{lemma on div with remainder}.

\begin{lemma} \label{lemma on div with remainder part 2} ($z$-division). Let $w=w(z,\sigma ),\ r=r(z,\sigma )$ be nonzero elements in $A_q$. Assume that $\deg _z w\leq \deg _z r$. Then there exist nonzero $g=g(\sigma )\in S$ and $h=h(z,\sigma )\in A_q$ such that
\begin{equation}\label{eq div with rem}
\deg_z (gr-hw)<\deg _z w.
\end{equation}
In case the coefficient of the highest (or the lowest) power of $z$ in $w$ is a unit, one may choose $g(\sigma)$ to be a unit as well.
\end{lemma}

\section{The category $\mathcal{M}$}

\subsection{The definition and first properties of objects in the category $\mathcal{M}$}

By an $A_q$-module we will always mean a {\bf left} $A_q$-module. 

\begin{defn} Let $A_q\text{-Mod}$ be the abelian category of left $A_q$-modules, and $\mathcal{M}=\mathcal{M}_q\subset  A_q\text{-Mod}$ its full abelian subcategory consisting of modules which are {\bf finitely generated as $A$-modules}. We regard $\mathcal{M}$ as the category of ``coherent sheaves" on the algebraic quantum elliptic curve with the parameter $q\in \bbC ^*$.
\end{defn}

\medskip

\noindent{\bf Question 1.}
 Let $q=e^{2\pi i\alpha}$, $q'=e^{2\pi i\alpha'}$. Suppose that $\alpha '=\alpha ^{-1}$. Is then $\mathcal{M}_q\simeq \mathcal{M}_{q'}$?

\medskip

\begin{lemma} \label{any obj is a-free} Any object in $\mathcal{M}$ is a free (finitely generated) $A$-module.
\end{lemma}

\begin{proof} Let $M\in \mathcal{M}$ and denote by $M_{\tor}\subset M$ its $A$-torsion submodule. Then $\sigma (M_{\tor})=M_{\tor}$.
Hence the support $\Supp\,M_{\tor}\subset \Spec\,A=\bbC ^*$ is invariant under multiplication by $q$. However $M_{\tor}$ is a finitely generated torsion module, so its support consists of a finite number of points.
As  $q$ is not a root of unity,  $\Supp\,M_{\tor}=\emptyset$, and $M_{\tor}=0$. It remains to note that any finitely generated torsion-free $A$-module is free.
\end{proof}

We immediately obtain the following corollary.

\begin{cor}\label{cor on Artinian} The abelian category $\mathcal{M}$ is Artinian. In other words, every object in $\mathcal{M}$ has finite length.
\end{cor}

\begin{cor}\label{cor on cyclic} Every object in $\mathcal{M}$ is a cyclic
$A_q$-module.
\end{cor}

\begin{proof} The algebra $A_q$ is simple, and $A_q$ has infinite length as a left $A_q$-module. Therefore by a theorem of Stafford,  every $A_q$-module of finite length is cyclic \cite[Ch 1, Thm 8.18]{Bj}.
\end{proof}

\begin{defn} We call a nonzero element $p(z,\sigma )=p_m(z)\sigma ^m+\cdots+p_n(z)\sigma ^n$ $\sigma $-{\bf good} if $p_m(z), p_n(z)\in A$ are units.

Similarly, we call an element $r(z,\sigma)=r_m(\sigma )z^m+\cdots+r_n(\sigma )z^n$ $z$-{\bf good} if $r_m(\sigma ),r_n(\sigma )\in S$ are units.
\end{defn}

\begin{remark}\label{rem on prod of good elements}  Let $p,r\in A_q$. Then the product $p\cdot r$ is $\sigma $-good (resp. $z$-good) if and only if $p$ and $r$ are so.
\end{remark}

\begin{lemma} \label{comp of rank} Let $I\subset A_q$ be a nonzero left ideal and let $M=A_q/I$ be the corresponding cyclic $A_q$-module. Then the $A$-rank $\rk_AM$ is the minimum of the $\sigma$-degrees of nonzero elements in $I$. (Similarly, the $S$-rank $\rk_SM$ is the minimum of the $z$-degrees of nonzero elements in $I$.) In particular, $\rk_AM,\ \rk_SM <\infty$.
\end{lemma}

\begin{proof} We prove only the first claim, the second being completely analogous.

Let $e\in M$ be the image of $1\in A_q$. The collection of elements $\{\sigma ^ke\}_{k\in \bbZ}$ generates the $A$-module $M$. Let $0\neq w(z,\sigma )=w_m(z)\sigma ^m+\cdots+w_n(z)\sigma ^n\in I$ be an element of minimal $\sigma $-degree $n-m$.
Any $A$-linear relation among $\sigma^m e, \sigma^{m+1} e,\ldots, \sigma^{n-1} e$ would give an element of $I$ of lower $\sigma$-degree than $n-m$, contrary to assumption,
so
$$\mathcal{B} := \{1\otimes \sigma^m e,\ldots, 1\otimes \sigma^{n-1} e\}$$
is an independent set in the $K(A)$-vector space $K(A)\otimes _AM$.
On the other hand, for any $k\in \bbZ$ we have
\begin{equation}
\label{induct up}
w_n(q^k z) \sigma^{n+k} = \sigma^k w_n(z) \sigma^n = -\sum_{i=m}^{n-1} \sigma^k w_i(z) \sigma^i = -\sum_{i=m}^{n-1}  w_i(q^kz) q^{i+k},
\end{equation}
so by induction on $k\ge 0$, all $1\otimes \sigma^{n+k} e$ lie in the span of $\mathcal B$. Likewise,
\begin{equation}
\label{induct down}
w_m(q^{-k}z) \sigma^{m-k} = \sigma^{-k} w_m(z) \sigma^m = -\sum_{i=m+1}^{n} \sigma^{-k} w_i(z) \sigma^i = -\sum_{i=m+1}^{n}  w_i(q^{-k}z) q^{i-k},
\end{equation}
so by induction on $k\ge 1$, all $1\otimes w_{m-k}(z)$ lie in  the span of
$\mathcal B$. It follows that $\mathcal B$ is a basis for $K(A)\otimes _AM$.

\end{proof}

\begin{lemma}\label{crit for being in m} Let $0\neq I\subset A_q$ be a left ideal and let $M:=A_q/I$ be the corresponding cyclic $A_q$-module. Then $M\in \mathcal{M}$ if and only if $I$ contains a $\sigma$-good element.

Similarly, $M$ is a finitely generated $S$-module if and only if $I$ contains a $z$-good element.
\end{lemma}

\begin{proof} Again, we prove only the first claim.

Let $e\in M$ be the image of $1\in A_q$. Let $p(z,\sigma)=p_m(z)\sigma ^m+\cdots+p_n(z)\sigma ^n\in I$ be a $\sigma $-good element, so $p_m(z),p_n(z)\in A$ are units. Then for any
$k\in \bbZ$ the element
$$\sigma ^kp(z,\sigma )=p_m(q^kz)\sigma ^{m+k}+\cdots+p_n(q^kz)\sigma ^{n+k}\in I$$
is also $\sigma $-good. By \eqref{induct up} and \eqref{induct down}, it follows that for any $l\in \bbZ$,
$$\sigma ^le\in \sum _{i=0}^{n-1}A\cdot \sigma ^ie.$$

Conversely, assume that $M\in \mathcal{M}$. Then for some $m\leq n\in \bbZ$, $M$ is generated as an $A$-module by $\sigma ^me,\sigma ^{m+1} e,\ldots,\sigma ^ne$. It follows that
the ideal $I$ contains elements
$$r=r_m(z)\sigma ^m+r_{m+1}(z)\sigma ^{m+1} +\cdots+r_n(z)\sigma ^n+\sigma ^{n+1}$$
and also
$$s=\sigma ^{m-1}+s_m(z)\sigma ^m+s_{m+1}(z)\sigma ^{m+1} +\cdots+s_n(z)\sigma ^n.$$
Then $r+s\in I$ is a good element.
\end{proof}

\begin{remark}\label{rem on princ ideals} Let $0\neq I\subset A_q$ be a principal ideal generated by $p(z,\sigma)$. Then $M:=A_q/I\in \mathcal{M}$ if and only if $p(z,\sigma)$ is $\sigma$-good. Similarly, $M$ is finitely generated over $S$ if and only if $p(z,\sigma)$ is $z$-good.
\end{remark}

\begin{defn} \label{defn of a good elt} An object $M\in \mathcal{M}$ is called {\bf good} if $M\simeq A_q/I$, where $I$ is a principal ideal (generated by a $\sigma$-good element).
\end{defn}

\subsection{The presentation of an object in $\mathcal{M}$}

\begin{prop} \label{present of objects in m} Let $I\subset A_q$ be an ideal such that $M:=A_q/I\in \mathcal{M}$. Then $I$ is either principal generated by a $\sigma$-good element (Remark \ref{rem on princ ideals}), or else it has 2 generators. More precisely, let $p(z,\sigma )\in I$ be any $\sigma $-good element and $0\neq w(z,\sigma )\in I$ be any element of the lowest $\sigma $-degree. Then $I$ is generated by $p$ and $w$.
\end{prop}

\begin{proof} Assume that $I$ is not principal. By Lemma \ref{crit for being in m} it contains a $\sigma $-good element $p\in I$. Let $0\neq w\in I$ be an element of the lowest $\sigma $-degree. Assume that such
a $w$ can be chosen to be $\sigma $-good. Then by Lemma \ref{lemma on div with remainder}, for any $r\in I$ there exist a unit $g\in A$ and an element $h\in A_q$ such that
$$\deg _\sigma (gr-hw)<\deg _\sigma w,$$
which means that $gr=hw$, hence $r\in A_q\cdot w$, i.e. the ideal $I$ is principal, a contradiction.

So $w$ is not $\sigma $-good. (In particular $\deg _\sigma p>\deg _\sigma w$.) This implies that $p\notin A_q\cdot w$, i.e. $I$ is not principal.

Let $J:=A_q\cdot p+A_q\cdot w\subset I$ be the ideal generated by $p$ and $w$. Then $N:=A_q/J\in \mathcal{M}$, hence is $A$-torsion free. Choose $r(z,\sigma )\in I$. Then $\deg _\sigma r\geq \deg _\sigma w$, and we can $\sigma$-divide $r$ by $w$ (without remainder). That is there exists $0\neq g(z)\in A$ and $h\in A_q$ such that $g\cdot r=h\cdot w$. This means that the image of $r$ in the module $N$ is $A$-torsion. Hence $r\in J$.
\end{proof}

\begin{example}\label{ex of two gener} The following example shows that the second possibility in Proposition \ref{present of objects in m} indeed occurs.
Consider the left ideal $J=A_q\cdot (\sigma -1)\subset A_q$ and the corresponding object $M=A_q/J \in \mathcal{M}$. Let $e\in M$ be the image of $1 \in A_q$. Then $M=A\cdot e$, $\sigma e=e$ and $M\simeq \mathcal{O}$. This object is simple and therefore generated by any of its nonzero elements. Choose the generator $f=(1+z)e\in M$. Let $I:=Ann(f)\subset A_q$. We have $pf=wf=0$, where
$$p=p(z,\sigma )=\sigma ^2-(q+1)\sigma +q,\quad \text{and}\quad
w=w(z+1)\sigma -(qz+1).$$
Then $w\in I$ is an element of the smallest $\sigma$-degree, and $p\in I$ is a $\sigma$-good element (of the smallest $\sigma$-degree).
By Proposition \ref{present of objects in m}, $I=A_q\cdot p+A_q\cdot w$.
We have the relation $(qz+1)p=(\sigma -q)w$, i.e. $p$ is $\sigma$-divisible by $w$.
\end{example}

Of course, in the above example, for any $n\in \bbZ$, the annihilator of the generator $z^ne$ is a principal ideal generated by the $\sigma $-good polynomial $\sigma -q^n$. This raises the question if every $0\neq M\in \mathcal{M}$ is good.

\medskip

\noindent{\bf Question 2.} Let $0\neq M\in \mathcal{M}$. Is there a generator $m$ of $M$, whose annihilator is a principal ideal (generated by a $\sigma$-good element)?

\medskip

%
%
%

\subsection{Analyzing the $A$-rank and $S$-rank of an object in $\mathcal{M}$} Let $0\neq I\subset A_q$ be a left ideal, $M:=A_q/I$. Recall (Lemma \ref{comp of rank}) that the $A$-rank $\rk_AM$ (resp. the $S$-rank  $\rk_SM$) is the smallest $\sigma $-degree (resp. $z$-degree) of a nonzero element $r(z,\sigma )\in I$. The following lemma will be important to us later.

\begin{lemma}\label{lemma on realiz of sigma rank} (1) Suppose that $I$ contains a $\sigma $-good element (i.e. $A_q/I\in \mathcal{M}$). Let $0\neq r(z,\sigma )\in I$ have the smallest $z$-degree. Then $r$ is $\sigma $-good.

(2) Similarly, suppose that $I$ contains a $z$-good element. Let $0\neq r(z,\sigma )\in I$ be an element of the smallest $\sigma$-degree. Then $r$ is $z$-good.
\end{lemma}

\begin{proof} We only prove (1), since the proof of (2) is completely analogous.

Assume not. Choose a $\sigma$-good element $p=p(z,\sigma)\in I$. By assumption $\deg _zp\geq \deg _zr$. Then we can $z$-divide $p$ by $r$ (Lemma \ref{lemma on div with remainder part 2}). That is there exists $0\neq g(\sigma )\in S$ and $h(z,\sigma )\in A_q$ such that
$gp=hr$. However, the element $gp$ is $\sigma $-good and $hr$ is not (since $r$ is not $\sigma$-good). A contradiction.
\end{proof}

\section{Tensor structure and inner $\mathcal{H}om$ in the category $\mathcal{M}$}\label{sect def of tensor str}   In this section we denote the $A$-rank of an $A_q$-module simply by $\rk\,M$. When no ring is indicated, tensor products in
this section should be understood to be over $A$.

The category $A_q\text{-Mod}$ has a symmetric tensor structure. Namely, given $M,N\in A_q\text{-Mod}$ the $A$-module $M\otimes _AN$ (where $am\otimes n=m\otimes an$) is also an $A_q$-module: $\sigma (m\otimes n)=\sigma (m)\otimes \sigma (n)$. The $A_q$-modules $M\otimes _AN$ and $N\otimes _AM$ are isomorphic via the map $m\otimes n\mapsto n\otimes m$.
The $A_q$-module $\mathcal{O}$ (Definition \ref{defn of o}) is the unit for this tensor product. This tensor structure restricts to the subcategory $\mathcal{M}$ and clearly $\rk\,M\otimes _AN=(\rk\,M)(\rk\,N)$.

For $M,N\in A_q\text{-Mod}$ the $A$-module $\Hom _A(M,N)$ is an $A_q$-module: $$\sigma(f)(m):=\sigma (f(\sigma ^{-1}m)).$$
We denote this $A_q$-module by $\mathcal{H}om (M,N)$. Clearly the bifunctor $\mathcal{H}om (-,-)$ preserves the subcategory $\mathcal{M}$.

Using the $A_q$-module $\mathcal{O}$ we define the duality (contravariant) functor
$$(-)^\vee :\mathcal{M}\to \mathcal{M}^{opp},\quad M\mapsto
\mathcal{H}om (M,\mathcal{O}).$$
It is an involution, i.e. $M^{\vee \vee}=M$ via the map
$$M\to M^{\vee \vee},\quad m\mapsto (f\mapsto f(m)).$$
For $M,N\in \mathcal{M}$ we have a natural functorial morphism of $A_q$-modules
$$\mathcal{H}om (M,N)\otimes _AM\to N,\quad f\otimes m\mapsto f(m).$$
In particular there is a canonical pairing - the {\it evaluation map} (a map of $A_q$-modules)
$$\ev=\langle -,-\rangle :M^\vee \otimes M \to \mathcal{O}.$$

The standard canonical isomorphism of $A$-modules
$$M^\vee \otimes M\to \mathcal{H}om(M,M)$$
is a map of $A_q$-modules. Hence the map of $A_q$-modules
$$\mathcal{O}\to \mathcal{H}om(M,M), \quad 1\mapsto id$$
gives the {\it coevaluation} morphism of $A_q$-modules
$$\coev:\mathcal{O}\to M^\vee \otimes M$$ such that
the composition
$$M=M\otimes \mathcal{O}\stackrel{\id\otimes \coev}{\xrightarrow{\hspace{32pt}}} M\otimes M^\vee \otimes M\stackrel{\ev\otimes \id}{\xrightarrow{\hspace{22pt}}}M$$
is the identity.

This makes $\mathcal{M}$ into a symmetric rigid monoidal category \cite{EGNO}.

The next lemma will be useful for us later.

\begin{lemma} \label{lemma useful later 1} (1) For any $M,N\in \mathcal{M}$ there is a canonical isomorphism of $\bbC$-vector spaces
$$\Hom _{A_q}(M,N)=\Hom _{A_q}(A,\mathcal{H}om (M,N)).$$

(2) For $M\in \mathcal{M}$ and $N\in A_q\text{-Mod}$ there is a canonical isomorphism of $A_q$-modules
$$\mathcal{H}om (M,N)=M^\vee \otimes N.$$

(3) For $M,N,K\in \mathcal{M}$ there is a canonical isomorphism of $A_q$-modules
$$\mathcal{H}om(K\otimes M,N)=\mathcal{H}om(K,\mathcal{H}om(M,N)).$$
In particular (also taking into account (2)) we have
$$(M\otimes N^\vee)^\vee =\mathcal{H}om(M\otimes N^\vee ,\mathcal{O})=\mathcal{H}om(M,N^{\vee \vee})=\mathcal{H}om(M,N)=M^\vee \otimes N.$$

(4) For any $M\in \mathcal{M}$ and $N\in A_q\text{-Mod}$ there is a canonical functorial isomorphism
of $\bbC$-vector spaces
$$\Hom _{A_q}(A_q\otimes M,N)=\Hom _A(M,N).$$
In particular the $A_q$-module $A_q\otimes M$ is projective.
\end{lemma}

\begin{proof} (1): Given a morphism $f\in \Hom _{A_q}(A,\mathcal{H}om(M,N))$ the map $f(1):M\to N$ is a morphism of $A_q$-modules. This gives the bijection.

(2) and (3): One only needs to check that the standard isomorphisms of $A$-modules are compatible with the action of $\sigma$.

(4) Given $f\in \Hom _{A_q}(A_q\otimes M,N)$ define the $A$-map $g:M\to N$ as $g(m):=f(1\otimes m)$.

To define a map in the other direction notice that
any element in $A_q\otimes M$ has a unique presentation as a finite sum  $\sum _i\sigma ^i\otimes m_i$. Also recall that the $A_q$-module structure on $A_q\otimes M$ is given by $\sigma (x\otimes m)=\sigma x\otimes \sigma (m)$.

For $g\in \Hom _A(M,N)$ define the corresponding $f\in \Hom _{A_q}(A_q\otimes M,N)$ as
$$f\Bigl(\sum _i\sigma ^i\otimes m_i\Bigr):=\sum _i\sigma ^i(g(\sigma ^{-1}(m_i))).$$
One checks that $f$ is indeed a morphism of $A_q$-modules and the correspondence $g\mapsto f$ is the inverse to the map $f\mapsto g$ defined above.
\end{proof}

\section{Important subcategories of the category $\mathcal{M}$}

Given an object $M\in \mathcal{M}$ we may consider it as an $S$-module.

\begin{defn} An object $M\in \mathcal{M}$ is {\bf free} (resp. {\bf torsion}, resp. {\bf torsion free}) if it is so as an  $S$-module.  Let $\mathcal{M}_{\free}\subset \mathcal{M}_{\tf}$ and $\mathcal{M}_{\tor}$ be the corresponding full subcategories of $\mathcal{M}$.
\end{defn}




\begin{lemma} \label{tor is a submodule} Let $M\in \mathcal{M}$ and let $M_{\tor}\subset M$ be its torsion $S$-submodule. Then $M_{\tor}$ is an $A_q$-submodule, i.e. $M_{\tor}\in \mathcal{M}_{\tor}$.
\end{lemma}

\begin{proof} It suffices to prove that $z^{\pm}M_{\tor}\subset M_{\tor}$.
But this is clear: if $V\subset M$ is an $S$-submodule which is a finite-dimensional $\bbC$-vector space, then so is $z^{\pm}V$.
\end{proof}

\begin{cor}\label{first ex seq for e}  For any $M\in \mathcal{M}$ we have the functorial short exact sequence in $\mathcal{M}$
$$0\to M_{\tor}\to M\to M_{\tf}\to 0,$$
where $M\in \mathcal{M}_{\tor}$ and $M_{\tf}\in \mathcal{M}_{\tf}$.
\end{cor}

\begin{lemma}\label{fin gen e-mod} Let $M\in \mathcal{M}$. Then
$M$ is finitely generated as an $S$-module if and only if $M\in \mathcal{M}_{\free}$.
\end{lemma}

\begin{proof} If $M\in \mathcal{M}_{\free}$, then $M$ is a finitely generated $S$-module because $\rk_SM<\infty$ by Lemma \ref{comp of rank}.
Conversely, let $M$ be a finitely generated $S$-module. Then arguing as in Lemma \ref{any obj is a-free} (after interchanging the algebras $A$ and $S$) we find that $M$ is free over $S$.
\end{proof}

\section{Description of the category $\mathcal{M}_{\tor}$}

Let $S_{\fd}$ be the full subcategory of $S\text{-mod}$ consisting of $S$-modules finite-dimensional over $\bbC$.
It is a tensor category with respect to the tensor product over $\bbC$.
Via Jordan decomposition, $S_{\fd}$ is equivalent to the tensor category of finite dimensional $\bbC ^*$-graded vector spaces with a nilpotent operator. Let $\bbC _q\in S_{\fd}$ be the $1$-dimensional $S$-module
where $\sigma (v)=qv$. The tensor product functor $\bbC _q\otimes (-)$ is an autoequivalence of the abelian category $S_{\fd}$ which generates the action of the group $q^{\bbZ}$ on $S_{\fd}$. The quotient category $S_{\fd}/q^{\bbZ}$
(the {\it orbit category}) is equivalent to the tensor category of finite dimensional $\bbC ^*/q^{\bbZ}$-graded vector spaces with a nilpotent operator.

Recall that by definition the objects of the category $S_{\fd}/q^{\bbZ}$ are the same as objects in $S_{\fd}$ and the morphisms between $V$ and $W$ are
given by the direct sum
$$\bigoplus _{\gamma \in q^{\bbZ}}\Hom _{S_{\fd}}(V,\gamma (W).$$

We have the exact functor
\begin{equation} \theta :S_{\fd}\to \mathcal{M}_{\tor},\quad
V\mapsto A\otimes _{\bbC} V,
\end{equation}
where $\sigma (z^n\otimes v)=\sigma (z^n)\otimes \sigma (v)=q^nz^n\otimes \sigma (v)$.

\begin{lemma} \label{lemma on theta en equiv} The functor $\theta$ induces the equivalence of  categories
\begin{equation} \label{equiv of tensor cat}
\overline{\theta }:S_{\fd}/q^{\bbZ}\to \mathcal{M}_{\tor}
\end{equation}
\end{lemma}

\begin{proof} Let $M\in \mathcal{M}_{\tor}$. Then ${}_SM=\oplus M_\lambda$, where $M_\lambda \subset M$ is the generalized $\lambda$-eigenspace of $\sigma$, i.e. it consists of elements $m\in M$ that are annihilated by some power of $\sigma -\lambda$.

Choose a set $\Gamma \subset \bbC ^*$ of coset representatives for the subgroup $q^{\bbZ}\subset \bbC ^*$.

Let $m\in M_\lambda$. Then $\sigma (zm)=q\lambda zm$, i.e. $z(M_\lambda)=M_{q\lambda}$. It follows that the $A_q$-module $M$ is generated by its subspace
$$\delta (M):=\bigoplus _{\lambda \in \Gamma}M_\lambda.$$
Moreover $M$ is freely generated by $\delta (M)$ as an $A$-module. In particular, the space $\delta (M)$ is finite dimensional, i.e. $\delta (M)\in S_{\fd}$, and we have an isomorphism of $A_q$-modules
$\theta (\delta (M))\simeq M$. Therefore the functor $\theta$ is essentially surjective.

To show that $\theta $ induces a fully faithful functor $\overline{\theta }:S_{\fd}/q^{\bbZ}\to \mathcal{M}_{\tor}$, it suffices to notice that for $V,W\in S_{\fd}$ we have an isomorphism of vector spaces
$$\Hom _{A_q}(\theta (V),\theta (W))\simeq \Hom _S(V,\theta (W))=\bigoplus _{\gamma \in q^{\bbZ}}\Hom _{S_{\fd}}(V,\gamma (W))=\Hom _{S_{\fd}/q^{\bbZ}} (V,W).$$
\end{proof}

Lemma \ref{lemma on theta en equiv} (and its proof) have the following immediate consequences which we record for future reference.

\begin{cor} \label{cor that tor is a direct sum of fin dim} Let $M\in \mathcal{M}_{\tor}$. Then

(1) There is an isomorphism of $A_q$-modules $M=A\otimes _{\bbC}V$ for a finite dimensional $S$-module $V$.

(2) $M$ is a direct sum of its finite dimensional $S$-submodules  and every $\sigma$-eigenvalue in $M$ has finite multiplicity.

(3) In particular, the spaces of $\sigma$-invariants and $\sigma$-coinvariants of $M$ have the same finite dimension.
\end{cor}

\subsection{The category $\mathcal{M}_{\tor}$ is preserved by the duality $(-)^\vee :M\mapsto M^\vee$ and the tensor product}

\begin{lemma}\label{tor is pres by dual} The category $\mathcal{M}_{\tor}$ is preserved by duality. More precisely, if $M=A\otimes _{\bbC }V$ for a finite dimensional $S$-module $V$ (Corollary \ref{cor that tor is a direct sum of fin dim}), then
$M^\vee =A\otimes _{\bbC} V^*$.
\end{lemma}

\begin{proof} Indeed, we have the sequence of natural isomorphisms of $A_q$-modules
$$M^\vee =\Hom _A(A\otimes _{\bbC}V,A)=\Hom _{\bbC}(V,A)=A\otimes _{\bbC}\Hom _{\bbC}(V,\bbC)=A\otimes _{\bbC}V^*.$$
\end{proof}

\begin{lemma} \label{tor is pres by the tensor product} If $M,N\in \mathcal{M}_{\tor}$, then $M\otimes N\in \mathcal{M}_{\tor}$.
\end{lemma}

\begin{proof} If $M=A\otimes _{\bbC}V$ and $N=A\otimes _{\bbC}W$, then
$M\otimes _AN=A\otimes _{\bbC}(V\otimes _{\bbC}W)$.
\end{proof}

\begin{remark}
The above results show that the functor $\overline{\theta}$ in
Lemma \ref{lemma on theta en equiv} is an equivalence of rigid tensor categories.
\end{remark}

\section{Structure of the category $\mathcal{M}_{\tf}$}


\begin{remark} \label{remark on vanishing of hom} Clearly, $\Hom (\mathcal{M}_{\tor},\mathcal{M}_{\tf})=0$. Also it follows from Lemma \ref{fin gen e-mod} that
$$\Hom (\mathcal{M}_{\free},\mathcal{M}_{\tor})=0.$$
Indeed, for any $M\in \mathcal{M}_{\free}$, $N\in \mathcal{M}_{\tor}$ the image of a homomorphism $f:M\to N$ is an $A_q$-submodule which is finitely generated over $S$. By Lemma \ref{fin gen e-mod} this submodule is free over $S$, hence is zero.
\end{remark}

\subsubsection{Do the abelian categories $\mathcal{M}_{\tor}$ and $\mathcal{M}_{\free}$ generate $\mathcal{M}$?} One might hope that for any
$M\in \mathcal{M}_{\tf}$ there exists a short exact sequence in $\mathcal{M}$
$$0\to M_{\free}\to M\to M_{\tor}\to 0,$$
where $M_{\free}\in \mathcal{M}_{\free}$ and $M_{\tor}\in \mathcal{M}_{\tor}$.

This is not the case. We construct a counterexample in Lemma \ref{cor with counter ex}.

\begin{lemma} \label{cor with counter ex} Consider the element $s(z,\sigma) =z-(\sigma +\sigma ^{-1})\in A_q$ and let $$M:=A_q/A_q\cdot s.$$
Then

(1) $M$ is finitely generated over $A$ (i.e. $M\in \mathcal{M}$) and not finitely generated over $S$.

(2) $\rk_AM=2$ and $\rk_SM=1$.

(3) $M\in \mathcal{M}_{\tf}$.

(4) $M$ is a simple $A_q$-submodule.
\end{lemma}

\begin{proof}
(1) This follows from Remark \ref{rem on princ ideals}, since the polynomial $s$ is $\sigma $-good and not $z$-good.

(2) We have $\rk_AM=\deg_\sigma s=2$ and
$\rk_SM=\deg _zs=1$.

(3), (4) Since $M_{\tor}$ is an $A_q$-submodule of $M$ it suffices to prove that $M$ is simple, i.e. that $A_q\cdot s$ is a maximal left ideal.

Let $e\in M$ be the image of $1\in A_q$. It is clear that $\{e,\sigma e\}$ is an $A$-basis of $M$. The action of $\sigma $ in this basis is given by the formulas
\begin{equation}\label{formulas} e\mapsto \sigma e,  \quad \sigma e\mapsto qz\sigma e -e.
\end{equation}
Suppose that $M$ contains a proper $A_q$-submodule $N\subset M$. Then $\rk_AN=1$, i.e. $N$ is a line bundle. Therefore there exists $0\neq n\in N$ such that $(\sigma -cz^k)n=0$ for some $z\in \bbZ$ and $c\in \bbC ^*$ (see  \eqref{line bundle}). 
Suppose we have such an element $0\neq n\in M$, $n=p_1e+p_2\sigma e$ for $p_1,p_2\in A$. Then we have 
\begin{equation}\label{calc} \sigma n=cz^k(p_1e+p_2\sigma e)=p_1^\sigma \sigma e+p_2^\sigma (qz\sigma e-e),
\end{equation}
where $p^\sigma :=\sigma p \sigma ^{-1}$. Equating the coefficients of $e$ and $\sigma e$ in \eqref{calc} we find that 
$$cz^kp_1=-p_2^\sigma,\quad cz^kp_2=p_1^\sigma+qzp_2^\sigma.$$
Substituting $p_2=-cq^{-k}z^kp_1^{\sigma ^{-1}}$ into the second equation we get 
\begin{equation}\label{final}
-c^2q^{-k}z^{2k}p_1^{\sigma ^{-1}}+cqz^{k+1}p_1-p_1^\sigma =0.
\end{equation}

If $k>1$ and $l$ is the highest integer such that the $z^l$ coefficient of $p_1$ is non-zero, then the $z^{l+2k}$ coefficient in \eqref{final} is non-zero.
If $k=1$ and $l$ is the lowest integer such that the $z^l$ coefficient of $p_1$ is non-zero, then the $z^l$ coefficient in \eqref{final} in non-zero.
If $k=0$ and $l$ is the highest integer such that the $z^l$ coefficient of $p_1$ is non-zero, then the $z^{l+1}$ coefficient in \eqref{final} is non-zero.
Finally, if $k\le -1$ and $l$ is the lowest integer such that the $z^l$ coefficient of $p_1$ is non-zero, then the $z^{l+2k}$ coefficient in \eqref{final} in non-zero.
In every case, therefore, $p_1$ is non-zero, and this implies $p_2=0$, so
$n=0$, which completes the proof of the lemma.
\end{proof}

\subsection{The structure of  $M\in \mathcal{M}$ as an $S$-module}

The algebra $A=\bbC [z,z^{-1}]$ can be considered as the group algebra of the cyclic group $\Gamma$ with the generator $z$. This group acts on the algebra $S=\bbC [\sigma ,\sigma ^{-1}]$ via conjugation in $A_q$:
$$z(\sigma ):=z\sigma z^{-1}=q^{-1}\sigma.$$
It induces the $\Gamma$-action on $\Spec\, S=\bbC ^*$: $z(p)=qp$.
Let $N$ be an $S$-module. For $\lambda \in \bbC ^*$
denote by $N_\lambda \subset N$ the submodule consisting of elements annihilated by some power of the operator $\sigma -\lambda$. If $N$ is a torsion $S$-module, then $N=\oplus _{\lambda \in \bbC ^*}N_\lambda$. If moreover $N$ is an $A_q$-module, then $zN_\lambda =N_{q\lambda}$.

If $M\in \mathcal{M}_{\tor}$ then the structure of the $S$-module ${}_SM$ is given by Corollary \ref{cor that tor is a direct sum of fin dim}. The next proposition determines the $S$-module structure of $M\in \mathcal{M}_{\tf}$.

\begin{prop} \label{str of s-mod in m}  Assume that $M\in \mathcal{M}_{\tf}$. One can choose a free $S$-submodule $M_0\subset M$ of full rank, so that the $S$-module $M/M_0$ is torsion and has the following properties: $M/M_0=\oplus _{\lambda \in \bbC ^*}(M/M_0)_\lambda$, where each $(M/M_0)_\lambda$ is finite dimensional and the collection of $\lambda$'s for which $(M/M_0)_\lambda \neq 0$ is contained in a finite number of $\Gamma$-orbits.
\end{prop}

\begin{proof}  Let $M=A_q/I$ (Corollary \ref{cor on cyclic}). Let $w\in I$ be a nonzero polynomial of the smallest $z$-degree. Then $\rk_SM=\deg _zw$ (Lemma \ref{comp of rank}). Let $\deg_zw=k$. We may assume that
$$w =w_0(\sigma)+w_1(\sigma)z+\cdots+w_l(\sigma )z^k.$$
Let $e\in M$ be the image of $1\in A_q$ in $M$. The elements $e,z e,\ldots,z^{k-1}e$ are $S$-independent and we take $M_0:=\oplus _{i=0}^{k-1}Sz^ie$.

Denote by $M_{[m,n]}\subset M$ the $S$-span of the elements $\{z^ke\}_{k=m,\ldots,n}$, where $m$ may be $-\infty$ and $n$ may be $\infty$ (for example $M_0=M_{[0,k-1]}$). For any $t\in \bbZ$ we have
$$z^tw=w_0(q^{-t}\sigma )z^t+\cdots+w_k(q^{-t}\sigma )z^{t+k}\in I,$$
which means that $w_0(q^{-t}\sigma )M_{[t,t+k]}\subset  M_{[t+1,t+k]}$ and
$w_k(q^{-t}\sigma )M_{[t,t+k]}\subset  M_{[t,t+k-1]}$. This implies that the $S$-module $M$ is the sum of submodules $M_{[0,\infty]}$ and $M_{[-\infty ,k-1]}$ which have filtrations
\begin{equation}\label{first filt} M_0=M_{[0,k-1]}\subset M_{[0,k]}\subset M_{[0,k+1]}\subset\cdots,
\end{equation}
\begin{equation}\label{first filt} M_0=M_{[0,k-1]}\subset M_{[-1,k-1]}\subset M_{[-2,k-1]}\subset\cdots,
\end{equation}
where all subquotients are finite dimensional and the subquotient $M_{[0,n]}/M_{[0,n-1]}$ (resp. 
$M_{[-m-1,k-1]}/M_{[-m,k-1]}$) is annihilated by the polynomial $w_k(q^{k-n}\sigma )$ (resp.  by the polynomial $w_0(q^{m+1}\sigma )$). Therefore any point $\lambda \in \Spec\, S=\bbC^*$ is contained in the support of only finitely many subquotients of the two filtrations. It follows that $(M/M_0)_\lambda$ is finite dimensional for each $\lambda \in \bbC ^*$ and the set of $\lambda $'s for which $(M/M_0)_\lambda\neq 0$ is contained in the $\Gamma$-orbits of the roots of $w_0(\sigma)$ and $w_k(\sigma)$.
\end{proof}

\begin{remark} We believe that the conclusion of Proposition \ref{str of s-mod in m} holds for any choice of a free $S$-submodule $M_0\subset M$ of maximal rank.
\end{remark}

\subsection{The $S$-rank of a tensor product}

We do not know how to compute the $S$-rank of a tensor product $M\otimes N$ in general. However there is the following fact.

\begin{prop} \label{rank of tensor prod} Let $M,N\in \mathcal{M}$ and assume that $\rk_SM=0$, i.e. $M\in \mathcal{M}_{\tor}$. Then
$$\rk_S(N\otimes M)=\rk_S(N)\cdot \rk_A(M).$$
\end{prop}

\begin{proof} By Corollary \ref{cor that tor is a direct sum of fin dim},  $M=A\otimes _{\bbC}V$ for a finite dimensional $S$-module $V$. (In particular $\rk_A(M)=\dim _{\bbC}V$.)  Then
$$N\otimes _AM= N\otimes _AA\otimes _\bbC V=N\otimes _\bbC V,$$
i.e. the $S$-module $N\otimes _AM$ is the tensor product of the $S$-modules $N$ and $V$ with the $\sigma$-action
$$\sigma (n\otimes v)=\sigma (n)\otimes \sigma (v).$$
Let $N_0\subset N$ be a free $S$-submodule as in Proposition \ref{str of s-mod in m}, so that $N/N_0$ is a direct sum of finite dimensional $S$-modules.
Hence $(N/N_0)\otimes _\bbC V$ is also a direct sum of finite dimensional $S$-modules and therefore $\rk_SN=\rk_SN_0$ and $\rk_S(N\otimes _\bbC V)=\rk_S(N_0\otimes _\bbC V)$. So we may and will assume that $N=S$.

Notice that the $S$-module $V$ has a filtration with $1$-dimensional subquotients. Now it suffices to notice that for any $\alpha \in \bbC ^*$ there is an isomorphism of $S$-modules
$$\phi :S\to S\otimes _{\bbC}(S/(\sigma-\alpha)),\quad \phi(1)=1\otimes \bar{1}.$$
\end{proof}

\subsection{The category $\mathcal{M}_{\tf}$ is not preserved by the duality}\label{rem on m-tf not preserved by duality}

\begin{example} Consider $M\in \mathcal{M}_{\tf}$ such that $M=Ae_1\oplus Ae_2\in \mathcal{M}_{\tf}$ where $\sigma (e_1)=ze_1$ and $\sigma (e_2)=e_1+e_2$. Then we have the exact sequence in $\mathcal{M}$
\begin{equation}\label{counter ex seq}
0\to Ae_1 \to M\to M/Ae_1\to 0,
\end{equation}
where $Ae_1\in \mathcal{M}_{\free}$ and $M/Ae_1\simeq \mathcal{O}\in \mathcal{M}_{\tor}$. Applying the duality functor to the sequence \eqref{counter ex seq} we find that $\mathcal{O}^\vee \simeq \mathcal{O}$ (which is in $\mathcal{M}_{\tor}$) is a submodule of $M^\vee$.
\end{example}

Notice however that the module $M$ in the above example does not belong to $\mathcal{M}_{\free}$.

\section{Duality preserves the $S$-rank and the category $\mathcal{M}_{\free}$}\label{duality and s-mod structure}
In this section we prove the following two theorems.

\begin{thm}\label{thm duality preserves s-rank} Let $M\in \mathcal{M}$.
We have the equality $\rk_{S}M=\rk_{S}M^\vee$.
\end{thm}

\begin{thm} \label{the duality preserves free} The category $\mathcal{M}_{\free}$ is preserved by the duality.
\end{thm}

The following key proposition is the main step in the proof of both theorems. Recall the anti-automorphism $\epsilon :A_q\to A_q$ (Section \ref{intro of algebra}).

\begin{prop} \label{prop on equality of ranks} Let $t\geq 0$ and let $M=A_q/A_q\cdot p(z,\sigma )$, where $p(z,\sigma )$ is a $\sigma $-good polynomial
$$p(z,\sigma )=p_0(z)+p_1(z)\sigma +\cdots+p_{t-1}(z)\sigma ^{t-1}+\sigma ^t$$
($p_0(z)\in A_q$ is a unit). Put
\begin{equation}\label{proposed formula} r(z,\sigma ):= {}^\epsilon p(z,\sigma )\cdot p_0(z)^{-1}
\end{equation}
(so that $r(z,\sigma )$ is also a good polynomial). Then there exists an isomorphism of $A_q$-modules $A_q/A_q\cdot r(z,\sigma)\simeq M^\vee$.
\end{prop}

\begin{remark} Since $p_0(z)\in A_q$ is a unit, we also have an isomorphism of modules $$A_q/A_q\cdot {}^\epsilon p(z,\sigma )\simeq M^\vee.$$
However in the proof of Proposition \ref{prop on equality of ranks} it is more convenient to work with the polynomial $r(z,\sigma)$.
\end{remark}

\begin{proof}
Let $e\in M$ be the image of $1\in A_q$. Then the elements
$$e,\ \sigma e,\ldots,\sigma ^{t-1}e$$
form an $A$-basis of $M$.

Recall that we have the canonical perfect  pairing
$$\langle -,-\rangle :M^\vee \otimes _AM\to \mathcal{O}$$
which is also a morphism of $A_q$-modules (Section \ref{sect def of tensor str}). For $m\in M$ and $m^\vee \in M^\vee$ we
have $\langle \sigma m^\vee ,\sigma m\rangle =\sigma \langle m^\vee ,m\rangle$, which implies that
\begin{equation}\label{formula for pairing}
\langle \sigma ^{-k}m^\vee ,m\rangle =\sigma ^{-k}\langle m^\vee,\sigma ^km\rangle.
\end{equation}

Let $f\in M^\vee$ be an element such that
\begin{equation} \label{defn of f}\langle f,e\rangle =1\quad \text{and}
\quad \langle f,\sigma ^ke\rangle =0,\quad \text{for} \quad k=1,\ldots,t-1.
\end{equation}
This defines a morphism of $A_q$-modules $\phi :A_q\to M^\vee$ where $\phi(1)=f$. We claim that $\phi$ is surjective and the kernel of $\phi$ is the left ideal $A_q\cdot r(z,\sigma)$.

Since $\Bigl( \sum_{i=0}^t p_{t-i} \sigma^{t-i}\Bigr)e=0$,
for all $s\in\bbZ$ we have
\begin{equation}
\label{left recurrence}
0 = \langle f,\sigma^{s-t}\Bigl( \sum_{i=0}^t p_{t-i} \sigma^{t-i}\Bigr)e\rangle
=\sum_{i=0}^t p_{t-i}^{\sigma^{s-t}} \langle f,\sigma^{s-i}e\rangle.
\end{equation}

Let $X_{t,s}$ denote the set of ordered tuples of integers in $[1,t]$ whose coordinates sum to $s$ (which is the empty set if $s< 0$ and the one element set consisting of the empty tuple if $s=0$).
For $x = (x_1,\ldots,x_k)\in X_{t,s}$, let
$$\Pi_{x} := (-1)^k\prod_{j=1}^k p_{t-x_j}^{\sigma^{{x_1+\cdots+x_j}}}.$$
We claim that for $s>0$, we have
\begin{equation}
\label{right partition}
\sum_{i=0}^t p_{t-i}^{\sigma^s}\sum_{x\in X_{t,s-i}} \Pi_{x}=0.
\end{equation}
To see this, we express $X_{t,s}$ as the disjoint union over $1\le i\le t$ of tuples $(x_1,\ldots,x_k)$ with $x_k=i$.
Since
$$\Pi_{(x_1,\ldots,x_k)} = -p_{t-x_k}^{\sigma^{x_1+\cdots+x_{k-1}+x_k}}\Pi_{(x_1,\ldots,x_{k-1})},$$
we have
$$\sum_{x\in X_{t,s}} \Pi_{x} = \sum_{i=1}^{t} \sum_{(x_1,\ldots,x_{k-1})\in X_{t,s-i}}-p_{t-i}^{\sigma^s}\Pi_{(x_1,\ldots,x_{k-1})},$$
which implies the claim.

Likewise, for $s>0$, we have
\begin{equation}
\label{left partition}
\sum_{i=0}^t p_{t-i}^{\sigma^i}\sum_{x\in X_{t,s-i}} \Pi_{x}^{\sigma^i}=0.
\end{equation}
Indeed, expressing $X_{t,s}$ as the disjoint union over $1\le i\le t$ of tuples $(x_1,\ldots,x_k)$ with $x_1=i$,
since
$$\Pi_{(x_1,\ldots,x_k)} = -p_{t-x_1}^{\sigma^{x_1}}\Pi_{(x_2,\ldots,x_{k})}^{\sigma^{x_1}},$$
we have
$$\sum_{x\in X_{t,s}} \Pi_{x} = \sum_{i=1}^{t} \sum_{(x_2,\ldots,x_{k})\in X_{t,s-i}}-p_{t-i}^{\sigma^i}\Pi_{(x_2,\ldots,x_{k})}^{\sigma^i},$$
which implies \eqref{left partition}.

For $s\ge t$,
\begin{equation}
\label{left formula}
\langle f,\sigma^se\rangle = -p_0\sum_{x\in X_{t,s-t}} \Pi_{x}.
\end{equation}
For $s=t$, this is immediate from  \eqref{left recurrence} and the definition of $f$.
For $s>t$, it follows by induction from \eqref{left recurrence} and \eqref{right partition}.

\medskip

\noindent{\it Claim.} {\it For $s>t$, we have}
$$\Bigl\langle\Bigl(\sum_{i=0}^t \sigma^{-i}(p_{t-i}/p_0)\Bigr) f, \sigma^s e\Bigr\rangle = 0.$$

\begin{proof}
By \eqref{left partition},
$$\sum_{j=0}^t p_{t-j}^{\sigma^j}\sum_{x\in X_{t,s-j}} \Pi_{x}^{\sigma^j}=0,$$
and applying $\sigma^{-t}$ to both sides, we obtain
$$\sum_{j=0}^t p_{t-j}^{\sigma^{j-t}}\sum_{x\in X_{t,s-j}} \Pi_{x}^{\sigma^{j-t}}=0,$$
Substituting $j=t-i$, this implies
$$\sum_{i=0}^t p_{t-i}^{\sigma^{-i}}\Bigl(\sum_{x\in X_{t,s+i-t}}\Pi_{x}^{\sigma^{-i}}\Bigr) = 0.$$
By \eqref{left formula}, this implies
$$\sum_{i=0}^t (p_{t-i}/p_0)^{\sigma^{-i}}\sigma^{-i}(\langle f,\sigma^{s+i}e\rangle)=0,$$
which means
$$\Bigl\langle \Bigl(\sum_{i=0}^t (p_{t-i}/p_0)^{\sigma^{-i}}\sigma^{-i}\Bigr) f, \sigma^s e\Bigr\rangle = 0.$$
\end{proof}
Since
$$(p_{t-i}/p_0)^{\sigma^{-i}} \sigma^{-i} =\sigma^{-i} p_{t-i} p_0^{-1} = {}^\epsilon(p_{t-i}\sigma^i)p_0^{-1},$$
we have for any $s>t$
\begin{equation}
\langle ({}^\epsilon p)p_0^{-1} f, \sigma^s e\rangle = 0.
\end{equation}

As any $t$ consecutive terms of $\sigma^i e$ span $M$,
the proposition implies that left multiplication by $r = {}^\epsilon p \cdot p_0^{-1}$
annihilates $f$.  Thus, there is a map of left $A_q$-modules $A_q/r A_q\to M^\vee$ sending the residue class of $1$ to $f$.  This map is surjective since the matrix
$$\langle \sigma^{i-1}f, \sigma^{j-1}e\rangle_{1\le i,j\le t}$$
is unitriangular.  As $A_q/r A_q$ has $A$-rank $n$, it is isomorphic to $M^\vee$.
\end{proof}

\subsubsection{Proof of Theorem \ref{thm duality preserves s-rank}}

Let $M\in \mathcal{M}$ be a nonzero object, $M\simeq A_q/I$ for a left ideal $I\subset A_q$. We show that $\rk_SM^\vee \leq \rk_SM$.

Let $p\in I$ be a $\sigma$-good element such that $\rk_SM=\deg _zp$ (Lemma \ref{comp of rank} and Lemma \ref{lemma on realiz of sigma rank}). Put $N:=A_q/A_q\cdot p\in \mathcal{M}$ with the canonical surjection $N\twoheadrightarrow M$. We have
\begin{equation} \label{red of thm}
\rk_SM=\deg _zp=\rk _SN.
\end{equation}
Consider the dual inclusion $M^\vee \hookrightarrow N^\vee$. By Proposition
\ref{prop on equality of ranks} $N^\vee \simeq A_q/A_q\cdot r$ where  $$r(z,\sigma )=p_0(q^{-t}z)\cdot {}^\epsilon p\cdot  p_0(z)^{-1}.$$
Then
$$\rk_SN^\vee= \deg _zr =\deg _zp=\rk_SN.$$
Hence
$$\rk_SM^\vee \leq \rk_SN^\vee =\rk_SN=\rk_SM.$$
Applying the same argument to $M^\vee$ instead of $M$ we find that $\rk_SM=\rk_SM^\vee$, which proves Theorem \ref{thm duality preserves s-rank}.

\subsubsection{Proof of Theorem \ref{the duality preserves free}}

Let $M$ be a good module as in Proposition \ref{prop on equality of ranks}. In the notation of that proposition $M=A_q/A_q\cdot p$ and $M^\vee =A_q/A_q\cdot r$. Then
$$M\in \mathcal{M}_{\free} \ \Leftrightarrow \
                    \text{$p$ is $z$-good} \ \Leftrightarrow \
                   \text{$r$ is $z$-good} \ \Leftrightarrow \
                    M^\vee \in \mathcal{M}_{\free}.
$$
This proves Theorem \ref{the duality preserves free} for good modules.

Let $M\in \mathcal{M}_{\free}$ be any module. By Lemma \ref{crit for being in m} the ideal $I$ contains elements $x$ and $y$ which are $\sigma$-good and $z$-good respectively. For $n>0$ put
$$p:=(\sigma ^n+\sigma ^{-n})x+(z^n+z^{-n})y.$$
For $n\gg 0$ the polynomial $p\in I$ is both $\sigma$-good and $z$-good.

Let $N:=A_q/A_q\cdot p$. Then $N\in \mathcal{M}_{\free}$ and $M$ is a quotient of $N$, so $M^\vee \subset N^\vee$. Since $N$ is good, by the above argument we have $N^\vee \in \mathcal{M}_{\free}$. Hence also $M^\vee \in \mathcal{M}_{\free}$. This completes the proof of Theorem \ref{the duality preserves free}.

\section{The Picard group of the category $\mathcal{M}$}\label{sect on pic gr}

Line bundles over $A_q$ are by definition objects $L\in \mathcal{M}$
of $\rk_AL=1$. The collection of isomorphism classes of line bundles is an abelian group $\Pic(\mathcal{M})$ under the tensor product. Let us describe this group.

Let $L\in \Pic(\mathcal{M})$ and choose an isomorphism of $A$-modules $\phi :A\stackrel{\cong}{\to} L$. Then for the $\bbC$-linear automorphism $\sigma :L\to L$ we have
\begin{equation}
\label{line bundle}
\sigma (\phi (1))=cz^m\cdot \phi (1)
\end{equation}
for some $c\in \bbC ^*$ and $m\in \bbZ$. So $L\simeq A_q/A_q\cdot (\sigma -cz^m)$ and $\rk_SL=\vert m\vert$ (Lemma \ref{comp of rank}). If we change the isomorphism $\phi$, say consider $\psi :=dz^i\cdot \phi$, $d\in \bbC ^*$, $i\in \bbZ$, then
$$\sigma (\psi (1))=q^icz^m\cdot \psi (1).$$
This defines a bijection $\Pic (\mathcal{M})\to (\bbC ^*/q^{\bbZ})\times \bbZ$, which is clearly an isomorphism of groups. We call the integer $m$ in \eqref{line bundle} the {\bf degree} of $L$ and denote it $\deg (L)$. For $L\in \Pic(\mathcal{M})$ the dual module $L^\vee$ is the inverse line bundle $L^{-1}\in \Pic(\mathcal{M})$; in particular, $\deg(L^{-1})=-\deg(L)$. 

\section{Cohomology of objects in $\mathcal{M}$}

Given objects $M,N\in \mathcal{M}$ we denote by $\Ext ^i(M,N)$ the group computed in the ambient category $A_q\text{-Mod}$. Then by Lemma \ref{lemma list of basic prop},
$\Ext ^i(M,N)=0$ if $i\neq 0,1$. We do not know if these groups coincide with the ones computed in the category $\mathcal{M}$.

\begin{defn} For $M\in \mathcal{M}$ we define its cohomology groups $H^i(M)$ as $\Ext ^i(\mathcal{O},M)$.
\end{defn}

\begin{lemma} \label{lemma on coh} Given $M\in \mathcal{M}$ consider the complex (of $\bbC$-vector spaces)
$$K(M):0\to M\stackrel{(\sigma -1)\cdot}{{\xrightarrow{\hspace{22pt}}} }M\to 0.$$
Then $H^\bullet  (K(M))=H^\bullet (M)$, i.e.
$$H^0(M)=\ker (\sigma -1),\quad H^1(M)=\coker(\sigma -1).$$
\end{lemma}

\begin{proof} We have the projective resolution
\begin{equation}\label{res of o} 0\to A_q\stackrel{\cdot (\sigma -1)}{{\xrightarrow{\hspace{22pt}}} } A_q\to \mathcal{O}\to 0.
\end{equation}
Then for any $M$, the cohomology $H^i(M)$ is the cohomology of the complex
$$0\to \Hom _{A_q}(A_q,M)\stackrel{(\sigma -1)^*}{{\xrightarrow{\hspace{24pt}}} }\Hom _{A_q}(A_q,M)\to 0,$$
which is exactly the complex $K(M)$.
\end{proof}

\begin{remark} \label{rem that coh dep only on sigma} Lemma \ref{lemma on coh} shows the cohomology $H^\bullet(M)$ for $M\in \mathcal{M}$ depends only on the $S$-module structure of $M$.
\end{remark}

\subsection{The Riemann-Roch and Serre duality theorem for line bundles in $\mathcal{M}$}
\begin{theorem} \label{rr for line bundles} (1) $H^0(\mathcal{O})\simeq H^1(\mathcal{O})\simeq \bbC$.
(So the genus of the quantum elliptic curve is indeed $1$).

(2) Let $L\in \Pic(\mathcal{M})$, $\deg(L)=0$, $L\ncong \mathcal{O}$. Then $H^0(L)=H^1(L)=0$.

(3) (Riemann-Roch) Let $L\in \Pic(\mathcal{M})$, $\deg(L)\neq 0$. Then $H^0(L)=0$ and
$h^1(L)=\vert \deg(L)\vert$.

In particular, for any $L$ we have $h^0(L)\neq 0$ if and only if $L\simeq \mathcal{O}$ and
$$\chi (L)=h^0(L)-h^1(L)=-\vert \deg(L)\vert \leq 0$$

(4) (Serre duality) For any $L\in \Pic(\mathcal{M})$ we have $h^i(L)=h^i(L^{-1})$, $i=0,1$.
\end{theorem}

\begin{proof} (1) This is obvious from the complex $K(\mathcal{O})$.

(2) If $H^0(L)\neq 0$, then there exists a nonzero morphism $\mathcal{O}\to L$, which is necessarily an isomorphism. If $\deg(L)=0$, $L\ncong \mathcal{O}$, then one sees immediately from the complex $K(L)$, that the map $\sigma -1:L\to L$ is surjective, i.e. $H^1(L)=0$.

(3) Suppose $\deg(L)=d\neq 0$, (say $d>0$) and choose an isomorphism of $A$-modules $L\simeq A$. Under this isomorphism $\sigma (1)=cz^d$, $c\in \bbC ^*$. Then the images of  $1,z,\ldots, z^{d-1}\in L$ in $H^1(L)$ form a $\bbC$-basis.

(4) This follows from (1), (2), and (3).
\end{proof}

\subsection{Cohomology of a general object in $\mathcal{M}$}

\begin{lemma}\label{bound on h0} For any $M\in \mathcal{M}$ we have $h^0(M)\leq \rk_A(M)$.
\end{lemma}

\begin{proof} We claim the natural map of $A_q$-modules
$$\mathcal{O}\otimes _{\bbC}H^0(M)\to M$$
is injective. Indeed, if $m_j\in H^0(M)\subset M$ are linearly independent over $\bbC$ and $\sum c_{i,j} z^i m_j$ is zero, then left-multiplying by $\sigma^k$,
we obtain $\sum c_{i,j} q^{ik}z^i m_j = 0$ for all $k$. As $q$ is not a root of unity, this implies all $c_{i,j}$ are $0$.
\end{proof}

\begin{prop}\label{prop on coh of tf} Let $M\in \mathcal{M}_{\tf}$. Then $h^0(M)=0$ and $h^1(M)=\rk_SM$.
\end{prop}

\begin{proof} The first assertion is immediate from Lemma~\ref{lemma on coh}. Let us prove the second one.
By Lemma \ref{comp of rank} we know that $d:=\rk_SM<\infty$.

First assume that $S^{d}\simeq {}_SM\in \mathcal{M}_{\free}$. Consider the map $M\stackrel{\sigma -1}{{\xrightarrow{\hspace{18pt}}} } M$. Then $H^1(M)=\coker(\sigma -1)\simeq \bbC ^d$.

For a general $M\in \mathcal{M}_{\tf}$ we can find a free $S$-submodule $N\subset M$ of rank $d$ such that the $S$-module $M/N$ is torsion. (Note that $N$ and $M/N$ may not be $A_q$-modules.) Assume that we chose $N$ as in Proposition \ref{str of s-mod in m}, so that the torsion $S$-module $M/N$ has the form
$$M/N=\bigoplus _{\lambda \in \bbC}(M/N)_\lambda,$$
where $\sigma -\lambda$ acts nilpotently on $(M/N)_\lambda$ and each $(M/N)_\lambda$ is finite dimensional.

This gives the short exact sequence of $S$-modules
\begin{equation}\label{short exact of s modules}
0\to N\to M\to M/N\to 0.
\end{equation}

The cohomology $H^\bullet (M)$ depends only on the $S$-module structure of $M$ (Remark \ref{rem that coh dep only on sigma}). In fact, this cohomology in the category of $S$-modules coincides with the functor $\Ext^\bullet _S(\bbC ,M)$, where $\bbC =S/(\sigma -1)$. Therefore the sequence \eqref{short exact of s modules}
gives rise to the long exact sequence of cohomology
\begin{align*}
0&\to \Hom _S(\bbC ,N)\to \Hom _S(\bbC ,M) \to \Hom _S(\bbC ,M/N)\\
& \to \Ext ^1_S(\bbC ,N)\to \Ext ^1_S(\bbC ,M)\to \Ext ^1_S(\bbC ,M/N) \to 0.
\end{align*}
Clearly $\Hom _S(\bbC ,N) =  \Hom _S(\bbC ,M)=0$, and we have the exact sequence
\begin{equation}\label{4 term exact seq} 0\to
\Hom _S(\bbC ,M/N) \to \Ext ^1_S(\bbC ,N)\to \Ext ^1_S(\bbC ,M)\to \Ext ^1_S(\bbC ,M/N) \to 0
\end{equation}
As we showed above the space $\Ext^1_S(\bbC ,N)$ has dimension $d$. Hence the space $\Hom _S(\bbC ,M/N)$ is also finite dimensional.
The proposition now follows from the equality
$$\dim \Hom _S(\bbC ,M/N)=\dim\Ext ^1_S(\bbC ,M/N),$$
which is an immediate consequence of part (3)
of Corollary~\ref{cor that tor is a direct sum of fin dim}.

%
\end{proof}

\subsection{Riemann-Roch theorem for general objects in $\mathcal{M}$}

We summarize the results on cohomology of objects in $\mathcal{M}$ in the following theorem.

\begin{thm}\label{thm summary on cohomology} Fix $0\neq M\in \mathcal{M}$. Then

(1) $h^0(M),h^1(M)<\infty$, and $h^0(M)\leq \rk_AM$.

(2) If $M\in \mathcal{M}_{\tf}$, then $h^0(M)=0$ and $h^1(M)=\rk_S M>0$.

(3) $\chi (M):=h^0(M)-h^1(M)=-\rk_SM$. In particular, $\chi (M)=0$ if and only if $M\in \mathcal{M}_{\tor}$.
\end{thm}

\begin{proof} (1) and (2) follow from Lemma \ref{bound on h0} and
Proposition \ref{prop on coh of tf}.

(3) Consider the short exact sequence
$$0\to M_{\tor}\to M\to M_{\tf}\to 0,$$
where $M_{\tor}\in \mathcal{M}_{\tor}$ and $M_{\tf}\in \mathcal{M}_{\tf}$. Then $\chi (M)=\chi (M_{\tor})+\chi (M_{\tf})$. We have $\chi (M_{\tor})=0$ by part (3) of Corollary \ref{cor that tor is a direct sum of fin dim}. Also $\chi (M_{\tf})=-\rk_SM_{\tf}\leq 0$ by part (2) above. It remains to notice that $\rk_SM=\rk_SM_{\tf}$.
\end{proof}

\subsection{Serre duality for general objects in $\mathcal{M}$}

\begin{thm}\label{Serre duality}
(1) For any $M\in \mathcal{M}$ we have $\chi (M)=\chi (M^\vee)$.

(2) If $M\in \mathcal{M}_{\tor}$ or $M\in \mathcal{M}_{\free}$ then $h^i(M)=h^i(M^\vee)$ for $i=0,1$
\end{thm}

\begin{proof} (1) This follows from Theorem \ref{thm duality preserves s-rank} and part (3) of Theorem \ref{thm summary on cohomology}.

(2) If $M\in \mathcal{M}_{\tor}$, then $M=A\otimes _{\bbC}V$ for a finite dimensional $S$-module $V$ (Corollary \ref{cor that tor is a direct sum of fin dim}). Let $V=\oplus _\lambda V_\lambda$, where $V_\lambda$ is the generalized $\lambda$-eigenspace for $\sigma$. Then both $h^0(M)$ and $h^1(M)$ are equal to the number of Jordan blocks in $V_1$. By Lemma \ref{tor is pres by dual} we have $M^\vee =A\otimes V^*$ and $(V^*)_\lambda =(V_{\lambda ^{-1}})^*$. It follows that $h^i(M)=h^i(M^\vee)$ for $i=0,1$.

Now let $M\in \mathcal{M}_{\free}$.  Then $M^\vee \in \mathcal{M}_{\free}$ by Theorem \ref{the duality preserves free}.
By Proposition \ref{prop on coh of tf} we have $h^0(M)=0=h^0(M^\vee)$ and $h^1(M)=\rk_SM$, $h^1(M)=\rk_SM^\vee$. It remains to apply Theorem
\ref{thm duality preserves s-rank}.
\end{proof}

\section{The Euler form of the category $\mathcal{M}$}

Consider the Euler form on the category $\mathcal{M}$:
\begin{equation}\label{def of euler form} \chi (M,N):=\dim \Hom (M,N)-\dim \Ext ^1(M,N)
\end{equation}
(where $\Ext^1(-,-)$ is computed in the abelian category $A_q\text{-Mod}$).
Theorem \ref{Serre duality} implies that the Euler form is symmetric.

\begin{cor}\label{cor on symm eulet form} The Euler form \eqref{def of euler form} is symmetric, i.e. for any $M,N\in \mathcal{M}$ we have
$$\chi(M,N)=\chi(N,M).$$
\end{cor}

We need a lemma.

\begin{lemma}\label{assume lemma} For $M,N\in \mathcal{M}$ we have an isomorphism
$$H^\bullet (\mathcal{H}om (M,N))=\Ext ^\bullet(M,N).$$
\end{lemma}

Assume the lemma for now. We know that $\mathcal{H}om (M,N)=M^\vee \otimes N$  and  $(M^\vee \otimes N)^\vee =N^\vee \otimes M$ (Lemma \ref{lemma useful later 1}). So Theorem \ref{Serre duality} and Lemma \ref{assume lemma} imply that
$$\chi (M,N)=\chi (M^\vee \otimes N)=\chi(N^\vee \otimes M)=\chi (N,M),$$ which proves the corollary. So it remains to prove the lemma.

\begin{proof} Recall the standard projective resolution of the $A_q$-module
$\mathcal{O}=A$
\begin{equation} \label{stand res of a} 0\to A_q\stackrel{\cdot (\sigma -1)}{{\xrightarrow{\hspace{24pt}}} } A_q\to \mathcal{O}\to 0.
\end{equation}
We apply the exact functor $(-)\otimes _AM$ to the sequence \eqref{stand res of a} to get a resolution of $M$
\begin{equation} \label{stand res} 0\to A_q\otimes M\stackrel{\cdot ((\sigma -1)\otimes\id)}{{\xrightarrow{\hspace{44pt}}} } A_q\otimes M\to M\to 0.
\end{equation}
It follows from part (4) of Lemma \ref{lemma useful later 1} that the object $A_q\otimes M\in A_q\text{-Mod}$ is projective, so the complex
$$0\to\Hom _{A_q}(A_q\otimes M,N)\stackrel{((\sigma -1)\otimes id)^*}{{\xrightarrow{\hspace{48pt}}} } \Hom _{A_q}(A_q\otimes M,N)\to 0$$
computes $\Ext^\bullet (M,N)$. It also follows from part (4) of Lemma \ref{lemma useful later 1} (and its proof)
that it is isomorphic to the complex
$$0\to \mathcal{H}om(M,N)\stackrel{(\sigma -1)\cdot}{{\xrightarrow{\hspace{22pt}}} }
\mathcal{H}om(M,N)\to 0$$
which computes $H^\bullet (\mathcal{H}om(M,N))$ by Lemma \ref{lemma on coh}. This proves the lemma.
\end{proof}

\section{Some additional results about the category $\mathcal{M}$}\label{addit results}

The results of this section were suggested to us by Alexey Elagin.

\subsection{The derived category $D^b(\mathcal{M})$}
For $M,N\in \mathcal{M}$ we defined the groups $\Ext ^i(M,N)$ as the corresponding Ext-groups in the category $A_q\text{-Mod}$. However, these
groups coincide with the ones computed in the category $\mathcal{M}$.

\begin{prop} (i) The abelian category $\mathcal{M}$ is hereditary.

(ii) The natural functor $D^b(\mathcal{M})\to D^b_{\mathcal{M}}(A_q\text{-Mod})$ is an equivalence of categories.
\end{prop}

\begin{proof} (i) The category $A_q\text{-Mod}$ is hereditary and $\mathcal{M}$ is its thick abelian subcategory. Hence $\mathcal{M}$ is also hereditary by \cite{RVdB}, (Lemma A.1, Prop. A.2).

(ii) This is \cite{Br} (Th. 5.1) or \cite{Kr} (Prop. 4.4.17).
\end{proof}

\subsection{The quiver description of the category $\mathcal{M}$}

Let $\mathcal{S}\subset \mathcal{M}$ be the collection of all simple objects in $\mathcal{M}$. Consider the following quiver $Q$: the set of vertices is the set $\mathcal{S}$. For $M,N\in \mathcal{S}$ the
number of arrows from $M$ to $N$ is the dimension of the space $\Ext ^1(N,M)$.

\begin{remark} The quiver $Q$ is {\rm symmetric}, i.e. the number of arrows from $M$ to $N$ is the same as the number of arrows from $N$ to $M$. Indeed, if $M,N\in \mathcal{S}$ are non-isomorphic, then $\Hom (M,N)=\Hom (N,M)=0$, hence $\Ext^1(M,N)=\Ext^1(N,M)$ by Corollary \ref{cor on symm eulet form}.
\end{remark}

Let $\text{Mod-}\bbC Q$ be the category of all right modules over the
path algebra of the quiver $\bbC Q$. Let $\mathrm{mod}_0\text{-}\bbC Q\subset \text{Mod-}\bbC Q$ be the full abelian subcategory of finite dimensional
nilpotent representations. This is the thick abelian subcategory generated by $1$-dimensional (right) representations of $\bbC Q$.

\begin{prop} \label{quiver descr} There is a natural equivalence of abelian categories
$$ \mathcal{M}\to \tmod_0\text{-}\bbC Q$$
which takes a simple object $M\in \mathcal{M}$ to the $1$-dimenional module supported at the corresponding vertex of $Q$.
\end{prop}

\begin{proof} This is Theorem 4.12 and Corollary 4.14 in \cite{El}.
\end{proof}

Proposition \ref{quiver descr} has several applications.

\subsubsection{The group of auto-equivalences $\Aut(\mathcal{M})$} Recall that a right $\bbC Q$-module $P$ consists of the following data: to each vertex $v$ one assigns a vector space $P_v$ and for any two vertices $v,w$ one specifies a linear map $P_{vw}: \bbC ^{d_{vw}}\to \Hom _{\bbC}(P_w,P_v)$, where $d_{vw}$ is the number of arrows from $v$ to $w$.

Let $\Auteq(\text{Mod-}\bbC Q)$ be the group of auto-equivalences of the abelian category $\text{Mod-}\bbC Q$ (up to an isomorphism of functors).

Fix vertices $v,w$. Then there is a homomorphism
$$\alpha _{vw}:\GL_{d_{vw}}(\bbC)\to \Auteq(\text{Mod-}\bbC Q)$$
defined as follows: for a right $\bbC Q$-module $P$ the homomorphism $P_{vw}$ as above is replaced by the composition $P_{vw}\cdot g^{-1}$, for each $g\in \GL_{d_{vw}}$. And for right $\bbC Q$-modules $P,P'$ the map
$$\alpha _{vw}(g)\colon \Hom(P,P')\to \Hom(\alpha _{vw}(g)(P),\alpha _{vw}(g)(P'))$$ is the identity.

Clearly the auto-equivalences $\alpha _{vw}(g)$ preserve the subcategory
$\tmod_0\text{-}\bbC Q$.


The next proposition shows that the group of auto-equivalences of the abelian category $\mathcal{M}\simeq \tmod_0\text{-}\bbC Q$ is huge. This is in the contrast with the classical case: for an elliptic curve $E$ the group of auto-equivalences of the abelian category $\coh_E$ is generated by the automorphisms of the variety $E$ and by taking
the tensor product with a line bundle.

\begin{prop} For any vertices $v\neq w$ of $Q$ the kernel of the group
homomorphism $\alpha _{vw}:\GL_{d_{vw}}(\bbC)\to \Auteq(\tmod_0\text{-}\bbC Q)$ is contained in the center of $\GL_{d_{vw}}(\bbC)$.
\end{prop}

\begin{proof} We may assume that $d_{vw}\neq 0$.
Consider the collection $\mathcal{P}$ of objects in $\tmod_0\text{-}\bbC Q$ consisting of modules $P$ such that $P_w=\bbC=P_v$ and $P_z=0$ for any vertex $z\neq v,w$. We require the morphism
$$P_{vw}:\bbC ^{d_{vw}}\to \Hom (P_w,P_v)=\bbC$$
to be nonzero and $P_{wv}=0$.

Thus to define an object in $\mathcal{P}$ is the same as to give a nonzero vector in the dual vector space $(\bbC ^{d_{vw}})^*$. Two such modules are isomorphic if and only if the corresponding vectors are proportional.
Let $g\in \GL_{d_{vw}}(\bbC)$ be such that the functor $\alpha _{vw}(g):
\tmod_0\text{-}\bbC Q \to \tmod_0\text{-}\bbC Q$ is isomorphic to the identity. Then $\alpha _{vw}(g)$ acts trivially on the set of isomorphism classes of objects in $\mathcal{P}$, i.e. $g$ lies in the center of $\GL_{d_{vw}}(\bbC)$.
\end{proof}

\begin{remark} The previous proposition applies to any quiver.
\end{remark}

\begin{prop} The graph $Q$ is connected.
\end{prop}

\begin{proof} It follows from Theorem \ref{rr for line bundles} that the subgraph of $Q$ supported on the line bundles in $\mathcal{M}$ is connected. Let $M\in \mathcal{M}$ be a simple object, i.e. $M$ is a vertex in $Q$. If $M\notin \mathcal{M}_{\tor}$, then $\Ext^1(\mathcal{O},M)\neq 0$ by part (3) of Theorem \ref{thm summary on cohomology}. Suppose that $M\in \mathcal{M}_{\tor}$. Choose a line bundle $L$ of nonzero degree. Then
$\rk _S L\otimes M\neq 0$ (Proposition \ref{rank of tensor prod}). Hence
$$0\neq h^1(L\otimes M)=\Ext^1(L^{-1},M)$$
by part (3) of Theorem \ref{thm summary on cohomology},  part (2) of Lemma \ref{lemma useful later 1} and
Lemma \ref{assume lemma}.
\end{proof}





\end{document}